\documentclass[12pt]{amsart}

\usepackage[utf8]{inputenc}

\usepackage{amscd}
\usepackage{enumitem}
\usepackage{float}

\usepackage{verbatim}
\usepackage{hyperref}
\usepackage{blkarray}
\usepackage{amssymb}
\usepackage{amsmath}
\usepackage{amsthm}
\usepackage{graphicx}
\usepackage{tikz-cd}

\usepackage{tikz}
\usetikzlibrary{calc}
\usetikzlibrary{arrows.meta}

\usepackage{pdfsync}

\usepackage[bordercolor=white,backgroundcolor=gray!30,linecolor=black,colorinlistoftodos]{todonotes}

\newtheorem{fakt}{fakt}[section]

\newtheorem{Corollary}[fakt]{Corollary}

\newtheorem{Lemma}[fakt]{Lemma}

\newtheorem{Theorem}[fakt]{Theorem}

\theoremstyle{definition}

\newtheorem{Example}[fakt]{Example}

\newtheorem{Remark}[fakt]{Remark}

\newtheorem{Definition}[fakt]{Definition}

\newtheorem{Algorithm}[fakt]{Algorithm}

\newcommand{\N} {{\mathbb N}} \newcommand{\Z} {{\mathbb Z}} \newcommand{\Q} {{\mathbb Q}}  \newcommand{\C} {{\mathbb C}} \newcommand{\Pbb} {{\mathbb P}}

\newcommand{\Hilb}{\operatorname{HP}}

\newcommand{\Cl}{\operatorname{Cl}}
\newcommand{\Pic}{\operatorname{Pic}}

\newcommand{\Proj}{\operatorname{Proj}}

\newcommand{\Sym}{\operatorname{Sym}}
\newcommand{\Syz}{\operatorname{Syz}}

\newcommand{\Frob}[1]{\operatorname{F}^{{#1}*}}

\newcommand{\sign}{\operatorname{sign}}

\newcommand{\pos}{\operatorname{pos}}

\newcommand{\rank}{\operatorname{rank}}

\newcommand{\im}{\operatorname{im}}

\begin{document}
\title{Deciding Stability of Sheaves on Curves}

\author[Holger Brenner]{Holger Brenner}
\address{Holger Brenner\\
Institut für Mathematik\\
Universität Osnabrück\\
Albrechtstr. 28a\\
49076 Osnabrück
}
\curraddr{}
\email{hbrenner@uni-osnabrueck.de}
\thanks{}

\author[Jonathan Steinbuch]{Jonathan Steinbuch}
\address{Jonathan Steinbuch\\
Institut für Mathematik\\
Universität Osna\-brück\\
Albrechtstr. 28a\\
49076 Osnabrück}
\curraddr{}
\email{jonathan.steinbuch@uni-osnabrueck.de}
\thanks{}


\begin{abstract}
We give an algorithm to determine whether a kernel sheaf over a smooth projective curve over an algebraically closed field is semistable. 
The algorithm uses symmetric powers to make destabilizing subbundles visible as global sections. 
 \end{abstract}

\maketitle

\section*{Introduction}

Among the vector bundles (locally free sheaves) in algebraic geometry the semistable vector bundles are of particular interest, as they form suitable moduli spaces (see \cite{mumGIT},  \cite{huybrechtslehn}). So in some sense the generic bundle is semistable. However, for a concretely given vector bundle it is in general quite difficult to decide whether it is semistable or not.

In this paper we provide an algorithm to decide semistability for vector bundles on smooth projective curves over an algebraically closed field of characteristic $0$. 
Related work on determing the semistability of bundles has been done in
\cite[Proposition 5.1, Corollary 6.4]{BreLook}, \cite{kaidkasprowitz} for sheaves over $\Pbb^n$ and in \cite[Section 2]{brestab2013} for curves in positive characteristic.

Recall that the slope of a vector bundle $ \mathcal{F}$ is
$ \mu(  \mathcal{F}  )  = \frac{\deg{\mathcal{F}}}{\rank{\mathcal{F}}}$, the quotient between the degree and the rank of the bundle, and that $\mathcal{F}$ is called semistable if for 
all subbundles $\mathcal{E}\subseteq \mathcal{F}$ the inequality  $\mu(\mathcal{E}) \leq \mu(\mathcal{F})$ holds. The basic problem in deciding semistability is to control all subbundles at once.
Either we have to find a destabilizing subbundle or we have to exclude its existence. If the bundle has negative degree and a global section $\neq 0$, when the bundle is not semistable. Moreover, global sections can be computed with methods from computer algebra, at least if the bundle is given by a graded module of a homogeneous coordinate ring of the curve. The problematic case is when a bundle of negative degree has a subbundle of nonnegative degree but without global sections.
Our main observation is here that destabilizing subbundles are visible as global sections of suitable symmetric powers and twists of the bundle. For the precise statements see Theorem \ref{maintheoremsym} and Theorem \ref{maintheorem0}.

As we want to decide semistability of a vector bundle $\mathcal{F}$ computationally by looking at global sections of symmetric powers, the curve and the bundle have to be given in a computationally accessible way. We fix a normal standard graded ring $S$ of dimension $2$ and consider $X=\Proj S$ with its very ample invertible sheaf $\mathcal{O}_X(1)$. In this setting we look at syzygy sheaves (kernel sheaves) given as
$$0\longrightarrow \mathcal{F} \longrightarrow \bigoplus \mathcal{O}_X(-d_i) \overset{A}{\longrightarrow} \mathcal{O}_X \, ,$$ 
where $A$ is a matrix with one row and homogeneous elements of degrees $d_i$ as entries. This representation for a vector bundle is not a big restriction as we elaborate in Remark \ref{norestrictionremark}.

In Lemma \ref{explicitsymextmatrices} we describe how to compute symmetric powers of a kernel bundle itself as a kernel bundle, thus making Theorem \ref{maintheoremsym} working as an algorithm.
The algorithm is demonstrated in the examples in Section \ref{examplesection}. It turns out that even for quite simple examples the standard computer algebra programs are not strong enough to deal with global sections of very high symmetric powers. We describe our implementation details using linear algebra methods in the last Section. 

We thank Jan-Luca Spellmann and Alexandre Tchernev for valuable comments.

\section{Degree and global sections}

We recall some basic notions for bundles on smooth projective curves. All torsion-free sheaves on a smooth curve are locally free, see \cite[Theorem II.1.1.6]{okonekschneiderspindler}. We use the words vector bundles and locally free sheaf interchangeably, and we always assume that they have finite rank.

\begin{Lemma}\label{sheafsequenceinvertible}
Let $\mathcal{F}$ be a locally free sheaf of rank $r$ on a curve. There exists an invertible sheaf $\mathcal{L}$, a locally free sheaf $\mathcal{F}'$ of rank $r-1$ and an exact sequence:
$$0\longrightarrow  \mathcal{L}\longrightarrow \mathcal{F} \longrightarrow \mathcal{F}' \longrightarrow 0.$$
\begin{proof}
See \cite[Lemma 1.15]{teixidorbundles}.
\end{proof} 
\end{Lemma}

For a noetherian, smooth scheme $X$, we have a natural isomorphism between the divisor class group and the Picard group: $\Cl X \cong \Pic X$. As such we have a divisor class for every isomorphism class of invertible sheaves on $X$ and in particular if $X$ is a smooth projective curve over an algebraically closed field we can define the degree of an invertible sheaf $\mathcal{L} = \mathcal{L}(D)$ by the degree of the corresponding divisor.

The \emph{degree} of a locally free sheaf of rank $r$ over a smooth projective curve $X$ is defined as
$$\deg \mathcal{F} := \deg\bigwedge^r\mathcal{F} \, .$$  
Hence the degree of a bundle is given as the degree of an invertible sheaf. 

The \emph{slope} of a locally free sheaf $\mathcal{F}$ on $X$ is $\mu(\mathcal{F}) := \frac{\deg{\mathcal{F}}}{r}$. 
A sheaf is called \emph{stable} if its slope is larger than the slope of any proper subsheaf and \emph{semistable} if its slope is at least as big as that of any subsheaf. 
We call a subsheaf \emph{destabilizing} if it violates the stability condition.

The theorem of Riemann-Roch for sheaves relates the degree and rank to global sections and the genus $g$ of the curve:
$$\deg \mathcal{F} = \chi(\mathcal{F}) + r(g-1).$$
Here
$\chi(\mathcal{F})$ denotes the Euler-Poincaré characteristic, which on a smooth projective curve of genus $g$ is
$$\chi(\mathcal{F}) = \dim H^0(X,\mathcal{F}) - \dim H^1(X,\mathcal{F}) $$
and $H^0(X,\mathcal{F})$ is the vector space of global sections of $\mathcal{F}$. 
With the theorem of Riemann-Roch we get the following inequality:
$$\dim H^0(X,\mathcal{F}) \geq \deg \mathcal{F} - r(g-1).$$
From this follows immediately:

\begin{Lemma} \label{globalsectionroch}
Let $\mathcal{F}$ be a sheaf on a smooth projective curve $X$. 
If $\mu(\mathcal{F}) > g-1$ then $\mathcal{F}$ must have a nontrivial global section. 
\end{Lemma}

We will also use the following fact.

\begin{Lemma} \label{noglobalnegativesemistable}
Let $\mathcal{F}$ be a sheaf on a smooth projective curve $X$. 
If $\mu(\mathcal{F}) < 0 $ and $\mathcal{F}$ has a nontrivial global section, then $\mathcal{F}$ is not semistable. 
\end{Lemma}
\begin{proof}
Every nontrivial global section generates via $\mathcal{O}_X\rightarrow \mathcal{F}$ a subsheaf of nonnegative slope, which can not exist in a semistable sheaf of negative slope.
\end{proof}

We call such a nontrivial global section which shows that $\mathcal{F}$ is not semistable a \emph{destabilizing global section} of $\mathcal{F}$.

\section{Symmetric and exterior powers}

In this section we describe the multilinear operations we use.
We will use the sheaf versions of the tensor product, the symmetric power $\Sym^n$ and the exterior power $\bigwedge^n$. 
These are each the sheafifications of their respective module versions. 
We want to apply these operations to short exact sequences. Some versions of the following sequences are known, see \cite[Proposition 3]{sakaisymmetric}, \cite{lebelt}, \cite{avramovsymmetric}, but it is difficult to find an explicite source for what we need.

\begin{Lemma} \label{powersexactsequences}
Let $0 \rightarrow \mathcal{E} \overset{\psi}{\rightarrow} \mathcal{F} \overset{\varphi}{\rightarrow} \mathcal{G} \rightarrow 0$ be a short exact sequence of locally free sheaves over a scheme of characteristic $0$. This induces for every $n\in \N_{>0}$ exact sequences
\begin{align*}
0\longrightarrow \bigwedge^n\mathcal{E}\longrightarrow \bigwedge^n\mathcal{F} \longrightarrow & \left(\bigwedge^{n-1}\mathcal{F}\right)\otimes\Sym^1\mathcal{G} \longrightarrow \\ &\left(\bigwedge^{n-2}\mathcal{F}\right)\otimes\Sym^2\mathcal{G} \longrightarrow \\
& \cdots \longrightarrow\\
& \left(\bigwedge^2\mathcal{F}\right)\otimes\Sym^{n-2}\mathcal{G} \longrightarrow \\ & \left(\bigwedge^1\mathcal{F}\right)\otimes\Sym^{n-1}\mathcal{G}  \longrightarrow
 \Sym^n\mathcal{G}\longrightarrow 0.
\end{align*}
The leftmost map in this sequence is $\bigwedge^n\psi$. The maps in the middle of the sequence
$$\left(\bigwedge^k\mathcal{F}\right)\otimes\Sym^{n-k}\mathcal{G} \longrightarrow \left(\bigwedge^{k-1}\mathcal{F}\right)\otimes\Sym^{n-k+1}\mathcal{G},$$
are given by
\begin{multline*}f_1\wedge \ldots\wedge f_k  \otimes g_{k+1}\cdots g_n \mapsto \\ \sum_{i=1}^k(-1)^{i-1} f_1\wedge \ldots \wedge f_{i-1} \wedge f_{i+1} \wedge \ldots \wedge f_k \otimes \varphi(f_i)\cdot g_{k+1}\cdots g_n .
\end{multline*}
It also induces short exact sequences
\begin{align*}
0\longrightarrow \Sym^n\mathcal{E}\longrightarrow \Sym^n\mathcal{F} \longrightarrow & \left(\Sym^{n-1}\mathcal{F}\right)\otimes\bigwedge^1\mathcal{G} \longrightarrow \\ &\left(\Sym^{n-2}\mathcal{F}\right)\otimes\bigwedge^2\mathcal{G} \longrightarrow \\
& \cdots \longrightarrow\\
& \left(\Sym^2\mathcal{F}\right)\otimes\bigwedge^{n-2}\mathcal{G} \longrightarrow \\ & \left(\Sym^1\mathcal{F}\right)\otimes\bigwedge^{n-1}\mathcal{G}  \longrightarrow
 \bigwedge^n\mathcal{G}\longrightarrow 0.
\end{align*}
The leftmost map in this sequence is $\Sym^n\psi$. The maps in the middle of the sequence 
$$\left(\Sym^k\mathcal{F}\right)\otimes\bigwedge^{n-k}\mathcal{G} \longrightarrow \left(\Sym^{k-1}\mathcal{F}\right)\otimes\bigwedge^{n-k+1}\mathcal{G},$$
are given by
\begin{multline*}f_1 \cdots f_k  \otimes g_{k+1}\wedge \ldots\wedge g_n \mapsto \\ \sum_{i=1}^k f_1 \cdots f_{i-1} \cdot f_{i+1} \cdots f_k \otimes \varphi(f_i)\wedge g_{k+1}\wedge \ldots \wedge  g_n.
\end{multline*}
\begin{proof}
Since exactness is a local property we can assume we are working with free modules over a ring and that the original sequence is splitting, i.e. $\mathcal{F} \cong \mathcal{E}\oplus\mathcal{G}$.

We prove the exactness of the first sequence by induction over the rank $r$ of $\mathcal{E}$. 
For $r=0$ we have $\mathcal{F}\cong \mathcal{G}$ and the sequence is a well-known sequence of multilinear algebra (see for example \cite[ Aufgabe 86.20]{schejastorch}), where the maps are as stated as in our Lemma.

For $r\geq 1$ we fix an element $v$ in a basis of $\mathcal{E}$. 
Because $\mathcal{F}\cong \mathcal{E}\oplus\mathcal{G}$ it is also an element of a basis of $\mathcal{F}$. 
We write $\mathcal{E}=\langle v \rangle \oplus \mathcal{U}$, where $\mathcal{U}$ is a free module of rank $r-1$. Similarly we get $\mathcal{F}=\langle v \rangle \oplus \mathcal{U}'$.

We have $\bigwedge^k\mathcal{F} \cong \bigwedge^{k}\mathcal{U}'\oplus \left(\bigwedge^{k-1}\mathcal{U}'\right)\otimes \langle v\rangle$ by the map which concentrates every contribution of $v$ to the last component.
Because taking the tensor product with a free rank 1 module is an isomorphism we even have $\bigwedge^k\mathcal{F} \cong \bigwedge^{k}\mathcal{U}'\oplus \bigwedge^{k-1}\mathcal{U}'$. Similarly for $\bigwedge^k\mathcal{E}$.

This allows us to write the sequence as the direct sum of two sequences we know are exact by induction:

\begin{align*}
0\longrightarrow \bigwedge^n\mathcal{U}\longrightarrow \bigwedge^n\mathcal{U}' \longrightarrow & \left(\bigwedge^{n-1}\mathcal{U}'\right)\otimes\Sym^1\mathcal{G} \longrightarrow \\
\cdots \longrightarrow & 
\left(\bigwedge^1\mathcal{U}'\right)\otimes\Sym^{n-1}\mathcal{G}  \longrightarrow
\Sym^n\mathcal{G}\longrightarrow 0
\end{align*}

and

\begin{align*}
0\longrightarrow \bigwedge^{n-1}\mathcal{U}\longrightarrow \bigwedge^{n-1}\mathcal{U}' \longrightarrow & \left(\bigwedge^{n-2}\mathcal{U}'\right)\otimes\Sym^1\mathcal{G} \longrightarrow \\
\cdots \longrightarrow
& \left(\bigwedge^1\mathcal{U}'\right)\otimes\Sym^{n-2}\mathcal{G} \\
 \longrightarrow &\Sym^{n-1}\mathcal{G}\longrightarrow 0 \longrightarrow 0.
\end{align*}

Thus the sum sequence is exact as well. 
We want to prove that the map of the sum sequence is correct. 
Take a look at the diagram.
$$\begin{tikzcd}
\left(\bigwedge^{k}\mathcal{F}\right)\otimes\Sym^{n-k}\mathcal{G} \arrow[r] & \left(\bigwedge^{k-1}\mathcal{F}\right)\otimes\Sym^{n-k+1}\mathcal{G}\\
\begin{matrix}\left(\bigwedge^{k}\mathcal{U}'\right)\otimes\Sym^{n-k}\mathcal{G}\\
\oplus\\ \left(\bigwedge^{k-1}\mathcal{U}'\right)\otimes\Sym^{n-k}\mathcal{G}\end{matrix} \arrow[u] \arrow[r, shift left=6] \arrow[r, shift right=6] & \begin{matrix}\left(\bigwedge^{k-1}\mathcal{U}'\right)\otimes\Sym^{n-k+1}\mathcal{G} \\
\oplus\\ \left(\bigwedge^{k-1}\mathcal{U}'\right)\otimes\Sym^{n-k+1}\mathcal{G}\end{matrix} \arrow[u]
\end{tikzcd}$$

If we map an element $f_1\wedge\ldots\wedge f_k \otimes g + f'_1\wedge\ldots\wedge f'_{k-1} \otimes g'$ via the lower right route we get \begin{align*}&\sum_{i=1}^k(-1)^{i-1}f_1\wedge\ldots\wedge f_{i-1} \wedge f_{i+1} \wedge \ldots \wedge f_k\otimes \varphi(f_i) \cdot g \\
+ &\sum_{i=1}^{k-1}(-1)^{i-1}f'_1\wedge\ldots\wedge f'_{i-1} \wedge f'_{i+1} \wedge \ldots \wedge f'_{k-1}\wedge v \otimes \varphi(f'_i) \cdot g'.\end{align*} 
If we map via the upper left route we get \begin{align*}&\sum_{i=1}^k(-1)^{i-1}f_1\wedge\ldots\wedge f_{i-1} \wedge f_{i+1} \wedge \ldots \wedge f_k\otimes \varphi(f_i) \cdot g\\ + &\sum_{i=1}^{k-1}(-1)^{i-1}f'_1\wedge\ldots\wedge f'_{i-1} \wedge f'_{i+1} \wedge \ldots \wedge f'_{k-1}\wedge v \otimes \varphi(f'_i) \cdot g' \\+ &(-1)^{k-1}f'_1\wedge\ldots\wedge f'_{k-1} \otimes \varphi(v) \cdot g',\end{align*} but since $v\in \mathcal{E}$ we have $\varphi(v) = 0$, so the diagram commutes.

The second sequence works similarly. 
Again we start with a locally free sheaf $\mathcal{E}$ of rank $0$,  where the dual of the exterior power case gives us the sequence. 
For the dual sheaves we have the canonical isomorphisms $(\bigwedge^k\mathcal{F})^\vee \cong \bigwedge^k(\mathcal{F}^\vee)$ and as we work over a ring of characteristic $0$ we also have $(\Sym^k\mathcal{F})^\vee \cong \Sym^k(\mathcal{F}^\vee)$ \cite[Satz 83.7 and Satz 86.12]{schejastorch}.

We can also write the symmetric product of $\mathcal{F}\cong \mathcal{U}'\oplus \langle v \rangle$ as a direct sum as follows:
$$\Sym^k\mathcal{F}\cong \Sym^k\mathcal{U}'\oplus\Sym^{k-1}\mathcal{F}\otimes \langle v \rangle \cong \Sym^k\mathcal{U}'\oplus\Sym^{k-1}\mathcal{F}.$$
Notice how this time we have $\mathcal{F}$ itself in the second summand, so we have to do a double induction over the exponent $k$ and the rank. 
\end{proof}
\end{Lemma}

For us the most important part of these sequences are the maps $\bigwedge^n\mathcal{F} \rightarrow  \left(\bigwedge^{n-1}\mathcal{F}\right)\otimes\Sym^1\mathcal{G}$ and $\Sym^n\mathcal{F} \rightarrow  \left(\Sym^{n-1}\mathcal{F}\right)\otimes\bigwedge^1\mathcal{G},$ as they allow us to explicitly describe the exterior and symmetric powers of a kernel bundle as another kernel bundle, see Lemma \ref{explicitsymextmatrices}.

Because we will first take the exterior power and afterwards the symmetric power we will need the following two lemmas which show that we can in the same way describe the kernel of a sequence which is only left exact.
The two lemmas generalize \cite[Proposition 4.1]{kaidkasprowitz}. 

\begin{Lemma} \label{powerskernelsheaves}
Let $0 \rightarrow \mathcal{E} \rightarrow \mathcal{F} \rightarrow \mathcal{G}\rightarrow \ldots \rightarrow 0$ be an exact sequence of locally free sheaves over a scheme over a field of characteristic $0$, where $\varphi:\mathcal{F} \rightarrow \mathcal{G}$ is the second map.
Then $\bigwedge^n\mathcal{E}$ is the kernel of the map
\begin{multline*}\bigwedge^n\mathcal{F} \longrightarrow \left(\bigwedge^{n-1}\mathcal{F}\right)\otimes\mathcal{G}, \\ f_1\wedge \ldots\wedge f_n\mapsto \sum_{i=1}^n(-1)^{i-1} f_1\wedge \ldots \wedge f_{i-1} \wedge f_{i+1} \wedge \ldots \wedge f_k \otimes \varphi(f_i).\end{multline*}
	
Also, $\Sym^n\mathcal{E}$ is the kernel of the map
\begin{multline*}\Sym^n\mathcal{F} \longrightarrow \left(\Sym^{n-1}\mathcal{F}\right)\otimes\mathcal{G}, \\ f_1\cdots f_n\mapsto \sum_{i=1}^n f_1\cdots f_{i-1} \cdot f_{i+1} \cdots f_k \otimes \varphi(f_i).\end{multline*}
\begin{proof}
Because all sheaves in the sequence are locally free so are the kernels by induction starting from the right. 
We take the short exact sequence $0 \rightarrow \mathcal{E} \rightarrow \mathcal{F} \rightarrow \im \varphi \rightarrow 0$ and construct the sequences of Lemma \ref{powersexactsequences}. 

Let's look at the sequence for the exterior power. The kernel of the map $\bigwedge^n\mathcal{F} \rightarrow \left(\bigwedge^{n-1}\mathcal{F}\right)\otimes\im \varphi$ is $\bigwedge^n\mathcal{E}$. 
Because of the local freeness of the involved modules the map $\left(\bigwedge^{n-1}\mathcal{F}\right)\otimes\im \varphi \rightarrow \left(\bigwedge^{n-1}\mathcal{F}\right)\otimes\mathcal{G}$ is injective, so the kernel does not change if we concatenate with this map. 
The same is true for $\Sym^n$.
\end{proof}
\end{Lemma}

\begin{Lemma} \label{powerskernelsheavescurves}
Let $0 \rightarrow \mathcal{E} \rightarrow \mathcal{F} \rightarrow \mathcal{G}$ be an exact sequence of locally free sheaves over a smooth curve over a field of characteristic $0$, where $\varphi:\mathcal{F} \rightarrow \mathcal{G}$ is the second map.
Then $\bigwedge^n\mathcal{E}$ and $\Sym^n\mathcal{E}$ have the same description as in Lemma \ref{powerskernelsheaves}.
\begin{proof}
The image $\im \varphi \subseteq \mathcal{G}$ is torsion free as a subsheaf of $\mathcal{G}$, thus $\im \varphi$ is locally free. We take the short exact sequence $0 \rightarrow \mathcal{E} \rightarrow \mathcal{F} \rightarrow \im \varphi \rightarrow 0$ and construct the sequences of Lemma \ref{powersexactsequences}. 

The rest of the proof is the same as for Lemma \ref{powerskernelsheaves}.
\end{proof}
\end{Lemma}

\section{Rank, degree and slope for symmetric powers}

To work more easily with the degree and slope of the sheaves and their symmetric and exterior powers involved we present some rank and degree computations. First note that the degree and the rank are additive.

\begin{Lemma}\label{degranktor}
Let $\mathcal{E}$ and $\mathcal{F}$ be locally free sheaves over a smooth projective curve. 
We have $\deg(\mathcal{F}\otimes\mathcal{E}) = \rank\mathcal{E}\cdot\deg\mathcal{F}+\rank\mathcal{F}\cdot\deg\mathcal{E}$ and $\rank\mathcal{F}\otimes\mathcal{E} = \rank\mathcal{F}\cdot \rank\mathcal{E}$. 
\begin{proof}
\cite[Lemma 1.16]{teixidorbundles}.
\end{proof}
\end{Lemma}

\begin{Lemma} \label{degranksymext}
Let $\mathcal{F}$ be a locally free sheaf of finite $\rank{\mathcal{F}} \geq 1$ on a smooth projective curve over an algebraically closed field $K$ and $n\in \N_{>0}$.
We have $\rank\Sym^n\mathcal{F} = \binom{n + \rank\mathcal{F} - 1}{n}$ and for the degree we have $\deg\Sym^n\mathcal{F} = \binom{n + \rank\mathcal{F} - 1}{n-1}\deg\mathcal{F}.$

Also $\rank\bigwedge^n\mathcal{F} = \binom{\rank\mathcal{F}}{n}$ and $\deg\bigwedge^n\mathcal{F} = \binom{\rank\mathcal{F} - 1}{n-1}\deg\mathcal{F}.$
\begin{proof}
We want to compute this via induction over $n+\rank \mathcal{F}$. For $n = 1$ the assertions are clear. 
For $\Sym^n$ of a line bundle $\mathcal{L}$ use that $\Sym^n\mathcal{L} = \mathcal{L}^{\otimes n}$. 
So the assertions are also true for every sheaf of rank 1.

We apply Lemma \ref{sheafsequenceinvertible} to the dual of $\mathcal{F}$ and dualize again to get a short exact sequence $0\rightarrow\mathcal{U}\rightarrow \mathcal{F} \rightarrow \mathcal{L} \rightarrow 0$, where $\mathcal{U}$ has one rank less than $\mathcal{F}$ and $\mathcal{L}$ is a line bundle. 
Note that $\deg \mathcal{L} = \deg\mathcal{F}-\deg\mathcal{U}$.

We apply Lemma \ref{powersexactsequences} to the sequence $0\rightarrow\mathcal{U}\rightarrow \mathcal{F} \rightarrow \mathcal{L} \rightarrow 0$. 
Because of $\bigwedge^2\mathcal{L} = 0$ we get the short exact sequence 
$$0\longrightarrow \Sym^n \mathcal{U}\longrightarrow \Sym^n \mathcal{F} \longrightarrow \left (\Sym^{n-1}\mathcal{F}\right)\otimes \mathcal{L}\longrightarrow 0.$$

First we compute the rank of the symmetric powers by induction over $n+\rank{F}$. 
\begin{align*}
\rank \Sym^n\mathcal{F}
&= \rank \Sym^{n}\mathcal{U} + \rank \left (\mathcal{L} \otimes \Sym^{n-1}\mathcal{F}\right ) \\
&= \rank\Sym^{n}\mathcal{U} + \rank\Sym^{n-1}\mathcal{F}\\
&= \binom{n+\rank\mathcal{F}-2}{n} + \binom{n+\rank\mathcal{F}-2}{n-1}\\
&=\binom{n+\rank\mathcal{F}-1}{n}.
\end{align*}

Now we do induction over $n+\rank{F}$ for the degree.
\begin{align*}
\deg \Sym^n\mathcal{F}
&= \deg \Sym^{n}\mathcal{U} + \deg \left (\mathcal{L} \otimes \Sym^{n-1}\mathcal{F}\right ) \\
&= \deg\Sym^{n}\mathcal{U} + \deg \Sym^{n-1}\mathcal{F} + \rank\left(\Sym^{n-1}\mathcal{F}\right)\cdot\deg \mathcal{L}\\
&= \binom{n+\rank\mathcal{F}-2}{n-1}\deg(\mathcal{U}) + \binom{n+\rank\mathcal{F}-2}{n-2}\deg\mathcal{F}\\& + \binom{n+\rank\mathcal{F}-2}{n-1}\cdot (\deg\mathcal{F}-\deg\mathcal{U})\\
&= \binom{n+\rank\mathcal{F}-2}{n-2}\deg\mathcal{F} + \binom{n+\rank(\mathcal{F})-2}{n-1}\cdot\deg\mathcal{F}\\
&=\binom{n+\rank\mathcal{F}-1}{n-1}\deg\mathcal{F}.
\end{align*}

For $\bigwedge^n\mathcal{F}$ we consider the short exact sequence $$0\longrightarrow \bigwedge^n \mathcal{U}\longrightarrow \bigwedge^n \mathcal{F} \longrightarrow \left (\bigwedge^{n-1}\mathcal{U}^{n-1}\right)\otimes \mathcal{L}\longrightarrow 0.$$
The right map is on affine subsets given by concentration of the contributions of $\mathcal{L}$ on the last factor. 
For $n\geq 2$ we do induction over $\rank{F}$.
\begin{align*}
\rank \bigwedge^n\mathcal{F}
&= \rank \bigwedge^{n}\mathcal{U} + \rank \left (\mathcal{L} \otimes \bigwedge^{n-1}\mathcal{U}\right ) \\
&= \rank \bigwedge^{n}\mathcal{U} + \rank \bigwedge^{n-1}\mathcal{U}\\
&= \binom{\rank\mathcal{U}}{n} + \binom{\rank\mathcal{U}}{n-1}\\
&=\binom{\rank\mathcal{F}}{n}.
\end{align*}

And similarly for the degree.
\begin{align*}
\deg \bigwedge^n\mathcal{F}
&= \deg \bigwedge^{n}\mathcal{U} + \deg \left (\mathcal{L} \otimes \bigwedge^{n-1}\mathcal{U}\right ) \\
&= \deg \bigwedge^{n}\mathcal{U} + \deg \bigwedge^{n-1}\mathcal{U} + \rank\left(\bigwedge^{n-1}\mathcal{U}\right)\cdot\deg \mathcal{L}\\
&= \binom{\rank\mathcal{U}-1}{n-1}\deg\mathcal{U} + \binom{\rank(\mathcal{U})-1}{n-2}\deg\mathcal{U}\\
&+ \binom{\rank\mathcal{U}}{n-1}\cdot (\deg\mathcal{F}-\deg\mathcal{U})\\
&= \binom{\rank\mathcal{U}}{n-1}\deg\mathcal{U} + \binom{\rank\mathcal{U}}{n-1}\cdot (\deg\mathcal{F}-\deg\mathcal{U})\\
&=\binom{\rank\mathcal{F}-1}{n-1}\deg\mathcal{F}.
\end{align*}

\end{proof}
\end{Lemma}

Note that the rank of $\Sym^n(\mathcal{F})$ of a rank $1$ sheaf stays $1$ as we have used in the proof.

\begin{Corollary}
Let $\mathcal{F}$ and $\mathcal{E}$ be locally free sheaves on $X$ and $s\in \N_{> 0}$. 
We have 
$\mu(\bigwedge^s \mathcal{F}) = \mu(\Sym^s \mathcal{F}) = s\cdot\mu(\mathcal{F})$ and $\mu(\mathcal{F}\otimes\mathcal{E}) = \mu(\mathcal{F})+\mu(\mathcal{E})$.

\begin{proof}
This follows from Lemma \ref{degranktor} and Lemma \ref{degranksymext}.
\end{proof}
\end{Corollary}

\begin{Lemma}\label{powersemistable}
Let $\mathcal{F}$ be a locally free sheaf over a smooth projective curve over an algebraically closed field of characteristic $0$. 
Fix $n\in \N_{>0}, k\in \Z$. The following are equivalent:
\begin{enumerate}
\item $\mathcal{F}$ is semistable.
\item $\Sym^n\mathcal{F}$ is semistable.
\item $\mathcal{F}\otimes\mathcal{O}(k)$ is semistable.
\end{enumerate}

Also the following is true: If $\mathcal{F}$ is semistable so is $\bigwedge^n\mathcal{F}$.
\begin{proof}
Let first $\mathcal{F}$ be not semistable. 
Then a destabilizing subsheaf $\mathcal{E}\subset \mathcal{F}$ gives a destabilizing subsheaf $\Sym^n\mathcal{E}\subset \Sym^n \mathcal{F}$ and a destabilizing subsheaf $\mathcal{E}\otimes\mathcal{O}(k)\subset \mathcal{F}\otimes\mathcal{O}(k)$ (as $\mathcal{O}(k)$ is flat).

If however $\mathcal{F}$ is semistable, then it is shown (even for bundles over a normal projective variety in characteristic 0) in \cite[Corollary 3.2.10]{huybrechtslehn} that the symmetric and exterior powers are also semistable.
\end{proof}
\end{Lemma}

\section{Destabilizing subbundles and destabilizing sections}

We want to determine if $\mathcal{F}$ is semistable only by looking at global sections of twists of symmetric powers of $\mathcal{F}$. The following theorem is the main reason why our algorithm works.

\begin{Theorem} \label{maintheoremsym}
Let $\mathcal{F}$ be a locally free sheaf on a smooth projective curve $X$ of genus $g$ over an algebraically closed field of characteristic $0$ and $r:=\rank{\mathcal{F}}$. 
Then $\mathcal{F}$ is semistable if and only if there does not exist a nontrivial global section of $(\Sym^q\mathcal{F})\otimes\mathcal{O}(k)$, where $q = (g-1+\deg X)n+1$, $n=\frac{r(r-1)}{\gcd(r,\deg\mathcal{F})}$ and $k=\left\lceil \frac{-q\mu(\mathcal{F})}{\deg X} \right\rceil - 1$.
\begin{proof}
First assume that $\mathcal{F}$ is not semistable. 
Take a destabilizing subbundle $\mathcal{E}\subseteq \mathcal{F}$ of rank $s < r$.
Then 
$$\frac{\deg\mathcal{E}}{s} = \mu(\mathcal{E}) > \mu(\mathcal{F}) = \frac{\deg\mathcal{F}}{r}$$ 
and thus 
$$\mu(\mathcal{E})-\mu(\mathcal{F}) = \left(\frac{r\cdot\deg \mathcal{E}-s\cdot\deg \mathcal{F}}{r\cdot s}\right)\geq \frac{\gcd(r,\deg \mathcal{F})}{r(r-1)} = \frac{1}{n} > 0.$$ 

\begin{tikzpicture}[line width=0.4mm]
\path[->] (-2,0) edge (8,0);
\coordinate[label=-90:$0$] (O) at (0,0);
\path[-] (O) edge (0,0.2);
\coordinate[label=-90:$g-1$] (g) at (5,0);
\path[-] (g) edge ($(g)+(0,0.2)$);
\coordinate[label=90:$\mu(\mathcal{F})$] (F) at (1,0);
\path[-] ($(F)-(0,0.2)$) edge ($(F)+(0,0.2)$);
\coordinate[label=90:$\mu(\mathcal{E})$] (E) at (2,0);
\path[-] ($(E)-(0,0.2)$) edge ($(E)+(0,0.2)$);
\coordinate[label=-90:$\geq \frac{1}{n}$] (A) at (1.5,0);
\path[{<[scale=1.5]}-{>[scale=1.5]}, red] (1,0) edge (2,0);
\end{tikzpicture}

\begin{tikzpicture}[line width=0.4mm]
\path[->] (-2,0) edge (8,0);
\coordinate[label=-90:$0$] (O) at (0,0);
\path[-] (O) edge (0,0.2);
\coordinate[label=-90:$g-1$] (g) at (5,0);
\path[-] (g) edge ($(g)+(0,0.2)$);
\coordinate[label=90:$\mu(\Sym^q\mathcal{F})$] (F) at (2,0);
\path[-] ($(F)-(0,0.2)$) edge ($(F)+(0,0.2)$);
\coordinate[label=90:$\mu(\Sym^q\mathcal{E})$] (E) at (7.5,0);
\path[-] ($(E)-(0,0.2)$) edge ($(E)+(0,0.2)$);
\path[{<[scale=1.5]}-{>[scale=1.5]}, red] (2,0) edge (7.5,0);
\end{tikzpicture}

\begin{tikzpicture}[line width=0.4mm]
\path[->] (-2,0) edge (8,0);
\coordinate[label=-90:$0$] (O) at (0,0);
\path[-] (O) edge (0,0.2);
\coordinate[label=-90:$g-1$] (g) at (5,0);
\path[-] (g) edge ($(g)+(0,0.2)$);
\coordinate[label=90:$\mu(\Sym^q\mathcal{F}(k))$] (F) at (-0.25,0);
\path[-] ($(F)-(0,0.2)$) edge ($(F)+(0,0.2)$);
\coordinate[label=90:$\mu(\Sym^q\mathcal{E}(k))$] (E) at (5.25,0);
\path[-] ($(E)-(0,0.2)$) edge ($(E)+(0,0.2)$);
\path[{<[scale=1.5]}-{>[scale=1.5]}, red] (-0.25,0) edge (5.25,0);
\end{tikzpicture}

With this we calculate 
\begin{align*}  \mu\left(\left(\Sym^q\mathcal{E}\right)\otimes\mathcal{O}(k)\right)
& = q\mu(\mathcal{E})+k\deg X \\
& = q\mu(\mathcal{E})+\left(\left\lceil \frac{-q\mu(\mathcal{F})}{\deg X} \right\rceil - 1\right)\deg X\\
& \geq q(\mu(\mathcal{E})-\mu(\mathcal{F})) - \deg X\\
& \geq \frac{q}{n}-\deg X\\
& = \frac{(g-1+\deg X)n+1}{n}-\deg X\\
& = g-1+\frac{1}{n}\\
& > g-1.
\end{align*}
Thus $(\Sym^q\mathcal{E})\otimes\mathcal{O}(k)$ has a global section by Lemma \ref{globalsectionroch} and this is also a global section of $(\Sym^q\mathcal{F})\otimes\mathcal{O}(k)$.

Now let us assume that $\mathcal{F}$ is semistable. 
Then so is $(\Sym^q\mathcal{F})\otimes\mathcal{O}(k)$ by Lemma \ref{powersemistable}. 
We calculate its slope \begin{align*}\mu\left((\Sym^q\mathcal{F})\otimes\mathcal{O}(k)\right) &= q\mu(\mathcal{F})+k\deg X\\
&= q\mu(\mathcal{F})+\left(\left\lceil \frac{-q\mu(\mathcal{F})}{\deg X} \right\rceil - 1\right)\deg X\\
&< q\mu(\mathcal{F}) + \frac{-q\mu(\mathcal{F})}{\deg X}\deg X\\
&= 0.
\end{align*}
Being semistable with negative slope it can't have any nontrivial global sections by Lemma \ref{noglobalnegativesemistable}, so we are done.
\end{proof}
\end{Theorem}

We will also prove the following variant of this theorem, which allows for a smaller $q$ (by at least a factor $r-1$), but requires to compute the global sections of several exterior powers. 
In general it is unclear which of these is easier to compute. 
While Theorem \ref{maintheoremsym} seems simpler, there are cases where a (parallelized) implementation of Theorem \ref{maintheorem0} is much faster and easier (see Example \ref{ExteriorGoodExample} below). 
Note also that for rank $r = 2$ the theorems are identical.

\begin{Theorem} \label{maintheorem0}
Let $\mathcal{F}$ be a locally free sheaf on a smooth projective curve $X$ of genus $g$ over an algebraically closed field of characteristic $0$ and $r:=\rank{\mathcal{F}}$. Then $\mathcal{F}$ is semistable if and only if for every $s<r$ there does not exist a nontrivial global section of $\Sym^q(\bigwedge^s\mathcal{F})\otimes\mathcal{O}(k)$, where $q = (g-1+\deg X)n+1$, $n=\frac{r}{\gcd(r,s\cdot\deg\mathcal{F})}$ and $k=\left\lceil \frac{-qs\mu(\mathcal{F})}{\deg X} \right\rceil - 1$
\begin{proof}
First assume that $\mathcal{F}$ is not semistable. 
Take a destabilizing subbundle $\mathcal{E}\subseteq \mathcal{F}$ of rank $s < r$.
By taking the $s$-th exterior power we get $\bigwedge^s \mathcal{E} \subseteq \bigwedge^s \mathcal{F},$ where $\bigwedge^s \mathcal{E}$ is the determinant bundle of $\mathcal{E}$ and is invertible.
Then 
$$\frac{\deg\mathcal{E}}{s} = \mu(\mathcal{E}) > \mu(\mathcal{F}) = \frac{\deg\mathcal{F}}{r}$$ 
and thus 
$$s(\mu(\mathcal{E})-\mu(\mathcal{F})) = \left(\deg \mathcal{E}-\frac{s\deg \mathcal{F}}{r}\right)\geq \frac{\gcd(r,s\deg\mathcal{F})}{r} = \frac{1}{n} > 0.$$ 
With this we calculate
\begin{align*} \mu\left(\Sym^q\left(\bigwedge^s\mathcal{E}\right)\otimes\mathcal{O}(k)\right)
& = qs\mu(\mathcal{E})+k\deg X \\
& = qs\mu(\mathcal{E})+\left(\left\lceil \frac{-qs\mu(\mathcal{F})}{\deg X} \right\rceil - 1\right)\deg X\\
& \geq qs(\mu(\mathcal{E})-\mu(\mathcal{F})) - \deg X\\
& \geq \frac{q}{n}-\deg X\\
& = \frac{(g-1+\deg X)n+1}{n}-\deg X\\
& = g-1+\frac{1}{n}\\
& > g-1.
\end{align*}
Thus by Lemma \ref{globalsectionroch} the invertible sheaf $\Sym^q(\bigwedge^s\mathcal{E})\otimes\mathcal{O}(k)$ has a global section and this is also a global section of $\Sym^q(\bigwedge^s\mathcal{F})\otimes\mathcal{O}(k)$.

Now let us assume $\mathcal{F}$ is semistable. Then so is $\Sym^q(\bigwedge^s\mathcal{F} )\otimes\mathcal{O}(k)$ by Lemma \ref{powersemistable}. We calculate its slope \begin{align*}\mu\left(\Sym^q\left(\bigwedge^s\mathcal{F}\right)\otimes\mathcal{O}(k)\right) &= qs\mu(\mathcal{F})+k\deg X\\
&= qs\mu(\mathcal{F})+\left(\left\lceil \frac{-qs\mu(\mathcal{F})}{\deg X}  \right\rceil - 1\right)\deg X\\
&< qs\mu(\mathcal{F}) + \left(\frac{-qs\mu(\mathcal{F})}{\deg X}\right)\deg X \\
&= 0.
\end{align*}
Being semistable with negative slope it can not have any nontrivial global sections, so we are done.
\end{proof}
\end{Theorem}

\section{Syzygy sheaves}

We want to describe the algorithm deciding semistability for kernel (or syzygy) sheaves. Let $\mathcal{F}$ be a \emph{kernel sheaf} over a smooth projective curve $X=\Proj S$, where $S$ is a normal 2-dimensional standard graded domain, given by
$$ 0 \longrightarrow \mathcal{F} \longrightarrow \bigoplus_{i=1}^n\mathcal{O}_X(-e_i)  \overset{A}{\longrightarrow} \bigoplus_{j=1}^m\mathcal{O}_X(-d_j).$$

As $\mathcal{F}$ is a subsheaf of a free sheaf it is torsion free and thus, because $X$ is a smooth curve, already locally free. Mainly we deal with syzygy sheaves for ideals, i.e. the case where $m=1$ and $A$ is a single row matrix  $A=(f_1,\ldots, f_n)$. Then we denote the kernel as $\Syz(f_1,\ldots,f_n)$.

\begin{Remark}\label{norestrictionremark}
All vector bundles on a smooth projective curve $X=\Proj S$ are isomorphic to bundles described in this way, at least after twisting and change of coordinate ring. Let $\mathcal{F}$ be of rank $r$ and $\mathcal{O}_X (1) $ be given. Then there exists $\ell$ such that $ \det (\mathcal{F}\otimes \mathcal{O}_X(l) )$ is very ample. We work with $\mathcal{G}=\mathcal{F}\otimes \mathcal{O}_X(l)  $ instead and use its determinant as the new $\mathcal{O}_X(1)$. We have  a presentation (see \cite[Lemma 2.3]{brebounds})
$\mathcal{O}_X^{r+1}\rightarrow\mathcal{G}^{\vee} (m) \rightarrow 0$ for $m$ large enough, the kernel is a line bundle, which is some $ \mathcal{O}_X(d) $ by the determinant property. Dualizing back and twisting gives a syzygy representation for $\mathcal{G}$.
\end{Remark}

\begin{Remark}
For our method it is important to determine the rank and the degree of $\mathcal{F}$ in an exact sequence $0\rightarrow \mathcal{F} \rightarrow \bigoplus_{i=1}^n \mathcal{O}_X(-e_i) \overset{A}{\rightarrow} \mathcal{O}_X.$ 
Let $X=\Proj S, S=K[x_1,\ldots,x_n]/I$ normal and $J\subseteq S$ the ideal generated by the entries of $A\neq 0$. 
Because of the additivity the rank we have $\rank \mathcal{F} = n-1$. The degree can be harder to compute:

Let $\mathcal{L} = \im A\subseteq  \mathcal{O}_X$.
The degree of $\bigoplus_{i=1}^n\mathcal{O}_X(-e_i)$ is $ -\deg X\cdot\sum_{i=1}^n e_i$.
Because of the additivity of degrees we have $\deg{\mathcal{F}} = -\deg X \cdot\sum_{i=1}^n e_i -\deg\mathcal{L}$. 
If $A$ is surjective (which is true if and only if $J$ is $S_+$-primary), computing the degree is easy, as then $\deg\mathcal{L} = \deg\mathcal{O}_X = 0$.

Otherwise we compute $\deg\mathcal{L}$ with the Hilbert polynomial as follows. 
Observe that for large $n\in \N$ we have
\begin{align*}\dim H^0(X,\mathcal{L}\otimes\mathcal{O}_X(n)) &= \chi(\mathcal{L}\otimes\mathcal{O}_X(n)) \\
&= \deg(\mathcal{L}\otimes\mathcal{O}_X(n)) + 1-g \\
&= \deg(\mathcal{L}) + n\cdot \deg X + 1-g.
\end{align*} 
Now $\dim H^0(X,\mathcal{L}\otimes\mathcal{O}_X(n))$ is exactly the number of degree $n$ elements of $J$, i.e. the Hilbert polynomial of the module $J$ with indeterminate $n$: $\Hilb(J)$. 
We have 
\begin{align*}\Hilb(J) &= \Hilb(S) -\Hilb(S/J)\\
&= n\cdot\deg X +(1-g) - \Hilb(S/J).
\end{align*}
Thus we deduce $\deg \mathcal{L} = -\Hilb(S/J) = - \dim_K(S/J) $ and 
$$\deg{\mathcal{F}} = -\deg X \cdot\sum_{i=1}^n e_i  +\Hilb(S/J).$$
\end{Remark}

\begin{Lemma}\label{explicitpowers}
Let $X$ be a scheme, $q,n,m \in \N_{>0}$ $e_1,\ldots,e_n,d_1,\ldots,d_m\in \Z$. 
We have
$$\left(\bigoplus_{i=1}^n \mathcal{O}_X(-e_i)   \right)     \otimes \left( \bigoplus_{j=1}^m\mathcal{O}_X(-d_j) \right) = \bigoplus_{i=1}^n\bigoplus_{j=1}^m\mathcal{O}_X(-d_j-e_i).$$
$$\Sym^q\left(\bigoplus_{i=1}^n\mathcal{O}_X(-e_i)\right) = \bigoplus_{(a_1,\ldots, a_n)\in \N^n, \sum_{i=1}^{n}a_i = q}\mathcal{O}_X\left(-\sum_{i=1}^{n}a_i\cdot e_i\right),$$
$$\bigwedge^q\left(\bigoplus_{i=1}^n\mathcal{O}_X(-e_i)\right) = \bigoplus_{I\subseteq \{1,\ldots,n\}, \#I = q}\mathcal{O}_X\left(-\sum_{i\in I}e_i\right).$$
\begin{proof}
These are special cases of the multilinear behavior of direct sums.
\end{proof}
\end{Lemma}

\begin{Lemma}\label{explicitsymextmatrices}
Let $\mathcal{F} = \ker A$ be a kernel sheaf over a smooth projective curve $X=\Proj S$, where $A: \bigoplus_{i=1}^n\mathcal{O}_X(-e_i) \rightarrow \bigoplus_{j=1}^m\mathcal{O}_X(-d_j)$ sits in an exact sequence as in Lemma \ref{powerskernelsheavescurves}. 
Then $\Sym^q\mathcal{F} = \ker A_{q}$, where 
\begin{multline*}
A_{q}: \Sym^q\left(\bigoplus_{i=1}^n\mathcal{O}_X(-e_i)\right) = \bigoplus_{a\in I}\mathcal{O}_X(-a\cdot e) \longrightarrow \\ \left(\Sym^{q-1}\left(\bigoplus_{i=1}^n\mathcal{O}_X(-e_i)\right)\right)\otimes\bigoplus_{j=1}^m\mathcal{O}_X(-d_j) = \bigoplus_{(b,j)\in J}\mathcal{O}_X(-b\cdot e-d_j)
\end{multline*}
is given in the following way. We index the columns of $A_{q}$ by the set $I=\{a=(a_1,\ldots, a_n)\in \N^n : \sum_{i=1}^{n}a_i = q\}$ and the rows by the set $J=\{(b,j)=(b_1,\ldots, b_n,j)\in \N^{n+1} : \sum_{i=1}^{n}b_i = q-1, 1 \leq j\leq m\}$. 
The entries are 
$$A_{q,(b_1,\ldots, b_n,j),(a_1,\ldots, a_n)} = \begin{cases}
0, & \textrm{if } \exists i\in \{1,\ldots,n\} : b_i>a_i \\
a_{i^*}\cdot A_{j,i^*}, & \textrm{otherwise; $i^*$ unique s.t. } a_{i^*}>b_{i^*}.\end{cases}$$
Also, $\bigwedge^q\mathcal{F} = \ker A_{\bigwedge^q}$, where 
$$A_{\bigwedge^q}: \bigwedge^q\left(\bigoplus_{i=1}^n\mathcal{O}_X(-e_i)\right) \longrightarrow \left(\bigwedge^{q-1}\left(\bigoplus_{i=1}^n\mathcal{O}_X(-e_i)\right)\right)\otimes\bigoplus_{j=1}^m\mathcal{O}_X(-d_j).$$
We index the columns of $A_{\bigwedge^q}$ by the subsets $I\subseteq \{1,\ldots,n\}, \#I = q$ and the rows by the set of tuples $(J,j)$, where $J\subseteq \{1,\ldots,n\},\#J = q-1$ and $1 \leq j\leq m$. 
The entries are
$$ A_{\bigwedge^q,(J,j),I} = \begin{cases}
0, & \textrm{if } J\not\subset I \\
\sign(i^*,I)\cdot A_{j,i^*}, & \textrm{otherwise; $i^*$ unique s.t. } i^*\in I\setminus J.\end{cases}$$
Here $\sign(i^*,I)=(-1)^{\pos(i^*)},$ where $\pos(i^*)$ gives the position of $i^*$ in $I$, induced by the order of $\{1,\ldots,n\}$.
\begin{proof}
The matrices $A_{q}$ and $A_{\bigwedge^q}$ are the explicit descriptions of the maps in Lemma \ref{powerskernelsheavescurves} when the symmetric and exterior powers of direct sums of invertible sheaves are expressed as in Lemma \ref{explicitpowers}.
\end{proof}
\end{Lemma} 

\begin{Example}
Take a plane smooth curve and over it the map $\mathcal{O}(-3)^{\oplus 2}\oplus \mathcal{O}(-2) \rightarrow \mathcal{O}$ which is given by the matrix $A = \begin{pmatrix} x^3 & y^3 & z^2\end{pmatrix}.$
We look at the matrix for the second symmetric power as given by Lemma \ref{explicitsymextmatrices}:
$$A_2: \mathcal{O}(-6)^{\oplus 3} \oplus \mathcal{O}(-5)^{\oplus 2} \oplus \mathcal{O}(-6) \longrightarrow \mathcal{O}(-3)^{\oplus 2} \oplus \mathcal{O}(-2).$$ It has the following entries (above and to the left we have written the summands to which the respective columns and rows correspond):

$$\begin{blockarray}{rcccccc}
& \mathcal{O}(-6) & \mathcal{O}(-6) & \mathcal{O}(-5) & \mathcal{O}(-6) & \mathcal{O}(-5) & \mathcal{O}(-4)\\ 
\begin{block}{r(cccccc)}
\mathcal{O}(-3) \, \, \, & 2x^3 & y^3 & z^2 & 0 & 0 & 0 \\
\mathcal{O}(-3) \, \, \, & 0 & x^3 & 0 & 2y^3 & z^2 & 0 \\
\mathcal{O}(-2) \, \, \,  & 0 & 0 & x^3 & 0 & y^3 & 2z^2 \\
\end{block}
\end{blockarray} .$$
\end{Example}

With this background we can formulate the steps we have to follow.

\begin{Algorithm}\label{mainalgorithm}
This algorithm decides semistability following the me\-thod of Theorem \ref{maintheoremsym}.
\begin{enumerate}
\item Start with a smooth projective curve $X=\Proj S$ given by a normal domain $S$ and a map $\bigoplus_{i=1}^{n}\mathcal{O}_X(-e_i) \rightarrow \mathcal{O}_X$ described by a matrix $A \neq 0$.
\item Compute the genus $g$ and the degree of $X$ (with the Hilbert polynomial of $X$), $\rank(\ker A)=n-1$ and the slope $\mu(\ker A) = -\frac{\deg \ker A}{n-1}$. 
From this compute $q$ and $k$ as in Theorem \ref{maintheoremsym}.
\item Compute $A_{q}$ as in Lemma \ref{explicitsymextmatrices}. 
It's a map $\bigoplus_{a\in I}\mathcal{O}_X(-e_a') \rightarrow \bigoplus_{b\in J}\mathcal{O}_X(-d_b')$ for some finite index sets $I,J$ and some degrees as computed in the Lemma.
\item Compute the dimension $d$ of the vector space of global sections of the kernel of the $k$-th twist of $A_q$.
\item $\ker A$ is semistable if and only if $d = 0$.
\end{enumerate}
\end{Algorithm}

Because in practice computing the dimension of the kernel becomes a lot more resource-intensive for larger $q$ and the corresponding $k$ it is advisable to first try some lower powers $q'$ and the corresponding $k'$. If we are lucky the dimension of the kernel will already be nonzero in which case we would already know that the sheaf is not semistable.

If one wants to decide semistability with Theorem \ref{maintheorem0} instead of Theorem \ref{maintheoremsym} the only difference in the algorithm is to compute $(A_{\bigwedge,s})_{q}$ instead of just $A_q$, where $q,$ $s$ and $k$ have the values given in Theorem \ref{maintheorem0}.

Note that the computations take place over the coordinate ring $S$, since we have $S_e= \Gamma(X, {\mathcal O}_X(e))$, and can be made with any computer algebra system.

\section{Examples and Computations}
\label{examplesection}

In this section we will give several examples. 
In some of these examples we relate the outcome of our algorithm with more specific methods.

\begin{Example}
Let $X=\Proj S$, $S=\C[x,y,z]/(f)$, where $f$ is a homogeneous polynomial such that $S$ is normal and $X$ smooth. In this example we will consider syzygy sheaves of the form $\Syz(x^n,y^n,z^n)$, which have rank $2$ and degree $-3n\cdot \deg X$. For $n=1$ this is the restriction of the cotangent sheaf of ${\mathbb P}^2$ to the curve.

We first look at $f=x^4+y^3z+z^4$, for which $X$ has genus 3. 
Theorem \ref{maintheoremsym} tells us to look at the symmetric power for $q=7$.
For $n=1$ we have to look at $k=\left\lceil \frac{3qn}{2} \right\rceil - 1 = 10$.
There are no global sections of $\Sym^7(\Syz(x,y,z))(10)$, thus $\Syz(x,y,z)$ is semistable.
For $n=2$, we get $k=20$ and again semistability.
If we look further at $n\leq 10$, for $n=3,4,8,9$ we find destabilizing sections, but for $n=5,6,7,10$ the syzygy sheaves are again semistable as we do not find global sections of the seventh symmetric power in the twists given by Theorem \ref{maintheoremsym}.

We can also look at higher degree curves, for example for $f=x^{10}+y^{9}z+z^{10}$ and find that the syzygy sheaves $\Syz(x^n,y^n,z^n)$ for $n=1,2,3$ are semistable. 
For this curve we have to look at the symmetric power for $q=46$ and for $n=3$ the deciding twist is already $206$.
This means that the resulting matrices become quite large and the computations take a while. 
\end{Example}

\begin{Example}\label{deg4example}
We look at $S:=\C[x,y,z]/(x^4+y^3z+z^4)$, the corresponding smooth projective curve $X\subseteq \Pbb^2$ of genus 3 and the kernel sheaf $\mathcal{F} := \Syz(x^3,y^3,z^2)$ of the surjective map $$A:\mathcal{O}_X(-3)^{\oplus 2}\oplus\mathcal{O}_X(-2)\twoheadrightarrow \mathcal{O}_X, (a_1,a_2,a_3)\mapsto (a_1x^3+a_2y^3+a_3z^2).$$

$\mathcal{O}_X(-3)^{\oplus 2}\oplus\mathcal{O}_X(-2)$ has degree $(-3-3-2)\cdot \deg(X) = -8\cdot4=-32$ and rank $3$, and $\mathcal{O}_X$ has degree $0$ and rank $1$. 
Because the map is surjective we see that $\mathcal{F}$ has degree $-32$, rank $2$ and slope $-16$.

For the first three symmetric powers there are no destabilizing global sections, i.e. $\Gamma(X,\Sym^q\mathcal{F}(k))$ is empty in the twist given by $k=q\cdot4-1$ (the twist coming from Theorem \ref{maintheoremsym}) and lower. For $q=4$ and the corresponding $k=15$ we find destabilizing sections (by Theorem \ref{maintheoremsym} we have to go up to  $q = 7$ and $k=q\cdot4-1= 27$). 

The destabilizing section for $q=4$ can also be detected by the following consideration. We see immediately from the curve equation that for $q=1$ we have the global section $(x,z,z^2)$ of $\mathcal{F}(4)$. The sheaf has slope $0$, thus the proper subsheaf generated by the section shows that $\mathcal{F}$ is not stable. Since the section has a common zero in the point $(0,1,0)$, we may also conclude that it is not semistable. 

The relation with global destabilizing sections of the $4$th symmetric power is as follows. For every global section of $\mathcal{F}(4)$ we get a global section of $\Sym^4(\mathcal{F}(4))$ by taking all possible products of 4 factors (repetitions allowed) out of $x$, $z$ and $z^2$. 
Consider Lemma \ref{explicitpowers}.

Explicitly for the global section $(x,z,z^2)$ we get the global section $v = (x^4,x^3z,x^3z^2,x^2z^2,x^2z^3,x^2z^4,xz^3,xz^4,xz^5,xz^6,z^4,z^5,z^6,z^7,z^8)$. 
Again $\Sym^4(\mathcal{F})(16)$ has slope $0$. 
But we have $x^4 = -y^3z-z^4$, so all entries contain the factor $z$, which we can divide out. 
This way we get a new global section $$v' = (-y^3-z^3,x^3,x^3z,x^2z,x^2z^2,x^2z^3,xz^2,xz^3,xz^4,xz^5,z^3,z^4,z^5,z^6,z^7)$$ of $\Sym^4(\mathcal{F})(15)$.
But $\Sym^4(\mathcal{F})(15)$ has negative slope, so we have shown that $\mathcal{F}$ is not semistable.
\end{Example}

By Grothendieck's splitting principle every vector bundle on the projective line is a direct sum of twists ${\mathcal O}_{ {\mathbb P}^1} (n)$, and such a bundle is only semistable if all twists agree.
But even if $\Pbb^1$ is given as a smooth quadric it is not clear how to find global sections by only looking at the homogeneous coordinate ring. 	The restriction of $\Syz(x,y,z)$ to any smooth quadric $X\subset \Pbb^2$ is isomorphic to $\mathcal{L}^{-3}\otimes \mathcal{L}^{-3}$, hence semistable, where $\mathcal{L} \cong \mathcal{O}_{\Pbb^1}(1)$ under an isomorphism $\Pbb^1\cong X$, but $\mathcal{L}$ can not be seen by looking at the global sections of $\Syz(x,y,z)$ alone.

\begin{Example} \label{quadricsplittingexample}
The ring $S:=\C[x,y,z]/(x^2+y^2+z^2)$ describes a quadric curve. 
It is isomorphic to $\Pbb^1$ and thus has genus 0. 
On it $\Syz(x^2,y^2,xz,yz)$ is locally free of $\rank 3$.

If we twist by $3$ we find global sections, for example $(z,0,-x,0)$. 
This does not give a destabilizing subsheaf however, because the degree of  $\Syz(x^2,y^2,xz,yz)(3)$ is $(9-8)\cdot 2 > 0$. 
The twist by $2$ does not have global sections. 
Global sections in that twist would directly give destabilizing subsheaves. The algorithm tells us that higher symmetric powers have destabilizing sections and that this sheaf is not semistable.

Let's look at the situation from a different angle. As a sheaf on $\Pbb^1$ the syzygy sheaf $\Syz(x^2,y^2,xz,yz)(2)$ is a direct sum of three invertible sheaves and it has degree $-4$. 
So the only possibility without global sections is  $$\Syz(x^2,y^2,xz,yz)(2) = \mathcal{L}^{-1}\oplus\mathcal{L}^{-1}\oplus \mathcal{L}^{-2},$$ where $\mathcal{L}$ is the invertible sheaf of degree one on the quadric when seen as a $\Pbb^1$. 
Note that $\mathcal{L}^2 = \mathcal{O}_X(1)$.

From this direct sum decomposition we can already see that the second symmetric power $\Sym^2(\Syz(x^2,y^2,xz,yz))$ will have three invertible summands of highest degree namely $(\mathcal{L}^{-1}(-2))^2 = \mathcal{O}_X(-5)$ and the second exterior power $\bigwedge^2(\Syz(x^2,y^2,xz,yz))$ will have one such summand.
The twist of interest $k$ in these powers becomes $5$ and indeed, if we twist by $5$ the global section space of the symmetric power turns out to be $3$-dimensional and of the exterior power $1$-dimensional. 
These correspond - in accordance with the theorem - to destabilizing subsheaves.
\end{Example}

\begin{Example}\label{ExteriorGoodExample}
We have seen that any destabilizing subsheaf will be made visible by destabilizing global sections of high enough symmetric powers. 
But computing exterior powers can be very useful if there are destabilizing subsheaves of rank 2 or higher. 
For a concrete example consider $S:=\C[x,y,z]/(x^n+y^n+z^n)$ and $\mathcal{F} = \Syz(x^4+y^2z^2,y^4,z^4,x^7)$ (almost any combination of polynomials of degree 4,4,4 and 7 would do).
The sheaf $\mathcal{F}$ is not semistable, in fact the subsheaf $\mathcal{E} = \Syz(x^4+y^2z^2,y^4,z^4)$ is a destabilizing subsheaf of rank 2:
$$\mu(\mathcal{E}) = \frac{-12n}{2} > \frac{-19n}{3} = \mu(\mathcal{F}).$$
Because of $\mathcal{E}$'s shape as a rank 2 syzygy sheaf its second exterior power becomes $\mathcal{O}_X(-12)$ and has a global section if twisted by 12. 
However even for $n = 5$ the fourth symmetric power $\Sym^4(\mathcal{F})$ is the lowest symmetric power with a destabilizing global section. 
For $n=9$ the lowest symmetric power with a destabilizing global section is $\Sym^{10}(\mathcal{F})$. Consider Table \ref{qmintable}.

\begin{table}[h]
\begin{tabular}{|r||l|l|l|} \hline
	$n$ & Genus $g$ & $q_{\min}$ & $q$ as of Theorem \ref{maintheoremsym} \\ \hline
	1 & 0 & 1 & 1\\
	2 & 0 & 1 & 7\\
	3 & 1 & 2 & 7\\
	4 & 3 & 1 & 31\\
	5 & 6 & 4 & 61\\
	6 & 10 & 2 & 31\\
	7 & 15 & 8 & 127\\	
	8 & 21 & 2 & 169\\
	9 & 28 & 10 & 73 \\
	10 & 36 & 8 & 271 \\ \hline
\end{tabular}
\bigskip
\caption{\label{qmintable}Table detailing the situation of Example \ref{ExteriorGoodExample} for several $n$. 
$q_{\min}$ is the lowest power $q=q_{\min}$ for which $\Sym^{q}(\mathcal{F})$ has a destabilizing section. 
We computed this with our implementation of the algorithm.}
\end{table}
\end{Example}

\begin{Example}
Smooth curves $X$ in the projective plane have genus $g=\frac{(\deg X-1)(\deg X -2)}{2}$. 
So with these, we only get $g=1,3,6,10,\ldots$. 
But embedded in higher dimensional projective space we can find curves to work over of any genus. 

A smooth curve of type $(a,b)$ in $\Pbb^1\times\Pbb^1$ has genus $g = (a-1)(b-1)$. 
As an explicit example look at the curve $X$ given by the relation $f = x_0^3y_0(y_0+y_1) + x_1^3y_1(y_0+2y_1) = 0$ for $([x_0:x_1],[y_0:y_1])\in\Pbb^1\times\Pbb^1$. 
We can work with this in our algorithm by using the Segre-embedding given by $\C[z_{00},z_{01},z_{10},z_{11}] \rightarrow \C[x_0,x_1,y_0,y_1], z_{ij}\mapsto x_iy_j$.

$X$ is then given by generators of the ring-kernel of the Segre-embed\-ding $z_{11}z_{00}-z_{10}z_{01}$ together with the preimage of $(f)$, generated by $z_{00}^2(z_{00}+z_{01})+z_{10}z_{11}(z_{10}+2z_{11})$ and $z_{00}z_{01}(z_{00}+z_{01})+z_{11}^2(z_{10}+2z_{11})$. 
By the Jacobi-Criterion and looking at the Hilbert polynomial we see that $X$ is indeed a smooth curve of genus 2 and degree 5.
We have a curve of type $(2,3)$. It is smooth and thus with regard to the Segre-embedding it is projectively normal \cite[Exercise III.5.6(b)(3)]{hartshorne}, i.e. the ring with these three relations is normal.

As example sheaves we compute: $\Syz(z_{01}z_{11}+z_{00}^2,z_{01}+z_{11},z_{10}z_{00}+z_{01}^2,z_{10})$ over $X$ is semistable, but $\Syz(z_{01}z_{11}+z_{00}^2,z_{11},z_{10}z_{00}+z_{01}^2,z_{10})$ is not.
\end{Example}

\section{A variant in positive characteristic}

We needed characteristic 0 to ensure that the symmetric and exterior powers of semistable sheaves are again semistable. 
This is not true in positive characteristic.
However, in positive characteristic the Frobenius pullbacks allow us to construct a very similar algorithm. 

The \emph{Frobenius pullback} $\Frob{e}\mathcal{F}$ of a locally free sheaf $\mathcal{F}$ is the pullback of $\mathcal{F}$ by the $e$-th power of the Frobenius homomorphism $f\mapsto f^{p^e}$. 
There is a surjective map from $\Frob{e}\Syz(f_1,\ldots,f_n) \rightarrow \Syz(f_1^{p^e},\ldots,f_n^{p^e})$. 
It is even bijective if the $f_i$ are primary to the irrelevant ideal $S_+$. 
Thus global sections of both sheaves are the same.

We mention the following immediate lemma.
\begin{Lemma}\label{Frobeniuspowerslope}
Let $\mathcal{F}$ be a locally free sheaf over a scheme of characteristic $p$. Then $\deg\left(\Frob{e}\mathcal{F}\right) = p^e\cdot\deg(\mathcal{F})$ and $\mu\left(\Frob{e}\mathcal{F}\right) = p^e\cdot\mu(\mathcal{F})$.
\end{Lemma}

\begin{Definition}
Let $\mathcal{F}$ be a locally free sheaf over a smooth projective curve over an algebraically closed field of characteristic $p$. 
$\mathcal{F}$ is called \emph{strongly semistable} if every Frobenius pullback is semistable.
\end{Definition}

It follows directly that $\mathcal{F}$ is strongly semistable if and only if $\Frob{e}(\mathcal{F})$ is strongly semistable.

As already mentioned in the introduction, a positive characteristic version of Theorem \ref{maintheoremsym} was already proved in \cite[Lemma 2.1]{brestab2013}. 
Note that there is a small mistake in the statement of a corollary \cite[Lemma 2.2]{brestab2013} with regards to the necessary Frobenius pullback, which we have corrected in our statement of the theorem.

\begin{Theorem}
Let $\mathcal{F}$ be a locally free sheaf on a smooth projective curve $X$ over a field of characteristic $p$ and $r:=\rank{\mathcal{F}}$.

$\mathcal{F}$ is strongly semistable if and only if there does not exist a nontrivial global section of $\Frob{e}(\mathcal{F})\otimes\mathcal{O}(k)$, for every $e\in\N$ and $k=\left\lceil \frac{-p^{e}\mu(\mathcal{F})}{\deg(X)} \right\rceil - 1$.

$\mathcal{F}$ is semistable if there does not exist a nontrivial global section of $\Frob{e}(\mathcal{F})\otimes\mathcal{O}(k)$, for an exponent $e\in \N$ with $p^e \geq (g-1+\deg(X))n+1$, $n=\frac{r(r-1)}{\gcd(r,\deg\mathcal{F})}$ and $k=\left\lceil \frac{-p^{e}\mu(\mathcal{F})}{\deg(X)} \right\rceil - 1$.
\end{Theorem}
\begin{proof}
We first prove the second assertion, in the same way as in Theorem \ref{maintheoremsym}.
First assume that $\mathcal{F}$ is not semistable. 
Take a destabilizing subbundle $\mathcal{E}\subseteq \mathcal{F}$ of rank $s < r$.
Then 
$$\frac{\deg(\mathcal{E})}{s} = \mu(\mathcal{E}) > \mu(\mathcal{F}) = \frac{\deg(\mathcal{F})}{r}$$ 
and thus 
$$\mu(\mathcal{E})-\mu(\mathcal{F}) = \left(\frac{r\cdot\deg(\mathcal{E})-s\cdot\deg(\mathcal{F})}{r\cdot s}\right)\geq \frac{\gcd(r,\deg(\mathcal{F}))}{r(r-1)} = \frac{1}{n} > 0.$$ 
With this we calculate 
\begin{align*}  \mu\left(\Frob{e}\left(\mathcal{E}\right)\otimes\mathcal{O}(k)\right)
& = p^{e}\mu(\mathcal{E})+k\deg(X) \\
& = p^{e}\mu(\mathcal{E})+\left(\left\lceil \frac{-p^{e}\mu(\mathcal{F})}{\deg(X)} \right\rceil - 1\right)\deg(X)\\
& \geq p^{e}(\mu(\mathcal{E})-\mu(\mathcal{F})) - \deg(X)\\
& \geq \frac{p^{e}}{n}-\deg(X)\\
& \geq \frac{(g-1+\deg(X))n+1}{n}-\deg(X)\\
& = g-1+\frac{1}{n}\\
& > g-1.
\end{align*}
Thus $\Frob{e}(\mathcal{E})\otimes\mathcal{O}(k)$ has a global section by Lemma \ref{globalsectionroch} and this is also a global section of $\Frob{e}(\mathcal{F})\otimes\mathcal{O}(k)$.

Now to the first assertion.
If $\mathcal{F}$ is not strongly semistable, some pullback will not be semistable. 
So for some (higher) power $e$ the Frobenius pullback $\Frob{e}(\mathcal{F})$ will have a destabilizing global section.

Now let us assume $\mathcal{F}$ is strongly semistable. Then so is $\Frob{e}(\mathcal{F})\otimes\mathcal{O}(k)$. We calculate its slope \begin{align*}\mu\left(\Frob{e}\left(\mathcal{F}\right)\otimes\mathcal{O}(k)\right) &= p^{e}\mu(\mathcal{F})+k\deg(X)\\
&= p^{e}\mu(\mathcal{F})+\left(\left\lceil \frac{-p^{e}\mu(\mathcal{F})}{\deg(X)} \right\rceil - 1\right)\deg(X)\\
&< p^{e}\mu(\mathcal{F}) + \frac{-p^{e}\mu(\mathcal{F})}{\deg(X)}\deg(X)\\
&= 0.
\end{align*}
Being semistable with negative slope it can have no nontrivial global sections, so we are done.
\end{proof}

For finite fields it is possible to give a bound on the necessary exponent to ascertain strong semistability. This is because the family of semistable vector bundles of fixed rank and bounded degree is a bounded family, which is itself the basic result for the existence of moduli spaces for vector bundles. Over a finite field there are thus only finitely many isomorphism types of Frobenius pullbacks with suitable twists for a strongly semistable vector bundle. Hence there will be a repetition of the form $\Frob{e}(\mathcal{F})\otimes\mathcal{O}(k) \cong \mathcal{F}  $.
Thus it suffices to check a fixed Frobenius power to determine strong semistability. 
For computational purposes this is not very helpful though, because the required Frobenius power is very high.

Using exterior powers like in the characteristic 0 case allows us to state the following variant of the theorem.

\begin{Theorem}
Let $\mathcal{F}$ be a locally free sheaf on a smooth projective curve $X$ over a field of characteristic $p$ and $r:=\rank{\mathcal{F}}$.

$\mathcal{F}$ is semistable if for every $s<r$ there does not exist a nontrivial global section of $\Frob{e}(\bigwedge^s\mathcal{F})\otimes\mathcal{O}(k)$, where $e = \lceil\log_p((g-1+\deg(X))n+1)\rceil$, $n=\frac{r}{\gcd(r,s\deg(\mathcal{F}))}$ and $k=\left\lceil \frac{-p^{e}s\mu(\mathcal{F})}{\deg(X)} \right\rceil - 1$.

The same is true if we substitute the Frobenius power with the symmetric power in this theorem.
\end{Theorem}
\begin{proof}
For this direction the proof in Theorem \ref{maintheorem0} works fully for characteristic $p$ (for the Frobenius power as well as the symmetric power). A destabilizing section prohibits semistability.
\end{proof}

\begin{Remark}
It is possible to check semistability for the characteristic $0$ case by reduction modulo $p$.
This is based on the fact that any destabilizing subsheaf in characteristic $0$ would also occur as a destabilizing subsheaf in characteristic $p$.
Thus we could try primes $p$ until we find one for which the sheaf over characteristic $p$ is semistable.
Then we know that the corresponding characteristic $0$ sheaf is also semistable.
This process has the potential to be a computationally faster way to show semistability, because arithmetic modulo $p$ is faster.
This effect is diminished by the fact that we might have to try a lot of primes and that the degrees grow faster with Frobenius pullbacks.

Even more, if for all primes that we try the sheaves are not semistable, then we can not draw a conclusion about the semistability of the corresponding characteristic $0$ sheaf.
In particular with this method we can never decide semistability for a sheaf that is not semistable.
Still reduction modulo $p$ could be a useful part in an adaptive approach to determining semistability where you try different angles of attack at the same time.
\end{Remark}

\section{Implementation details}

We need to compute the global sections of $\left(\Sym^q\bigwedge^s\mathcal{F}\right)\otimes\mathcal{O}(k)$, for some $q,s\in \N_{\geq 0}, k\in \Z$ over a curve $X=\Proj S$, with $S$ a graded integrally closed algebra of finite type over $K$.

\begin{Remark} \label{degreematrix}
As laid out in Lemma \ref{explicitsymextmatrices} and Algorithm \ref{mainalgorithm} we construct a matrix $A_q$ which sits in an exact sequence $$0 \longrightarrow \Sym^q\mathcal{F}\otimes\mathcal{O}_X(k) \longrightarrow  \bigoplus_{a\in I}\mathcal{O}_X(k-a\cdot e) \overset{A_q}{\longrightarrow} \bigoplus_{(b,j)\in J}\mathcal{O}_X(k-b\cdot e-d_j).$$
We want to compute the dimension of the vector space of global sections of the kernel of $A_q$.
You will find the dimension as an entry in the Betti table of the module presented by $A_q$, for which there are implementations in many computer algebra systems.
For very simple cases we did this in Macaulay2\cite{M2} and CoCoA\cite{CoCoA}.
This approach proved to be inefficient for anything but the most simple examples.
It turned out to be way faster to only compute the correct degree case as follows.

We apply the global section functor to the exact sequence and get the exact sequence $$0 \longrightarrow \Gamma(X,\Sym^q\mathcal{F}\otimes\mathcal{O}_X(k)) \longrightarrow  \bigoplus_{a\in I}S_{k-a\cdot e} \overset{A_q}{\longrightarrow} \bigoplus_{(b,j)\in J}S_{k-b\cdot e-d_j}.$$

For a fixed $d\in \Z$ the elements of $S_d$ form a finite dimensional vector space, with a basis given by the degree $d$ monomials of $R$ which are not multiples of leading monomials of a Gröbner basis of the defining ideal of $S$ (fix any degree-respecting monomial order).

All we have to compute is thus the kernel of a matrix $B_q(k)$ (computed from $A_q$) over a field, a linear algebra problem.
\end{Remark}

\begin{Example}\label{degreematrixexample}
Look at the map $\mathcal{O}_X(-4)^3\oplus \mathcal{O}_X(-7) \rightarrow \mathcal{O}_X$ over $X=\Proj(\C[x,y,z]/(x^9+y^9+z^9))$ from Example \ref{ExteriorGoodExample} which is given by the matrix $A = \begin{pmatrix} z^2y^2 + x^4 & y^4 & z^4 & x^7 \end{pmatrix}$. 

In the first power Theorem \ref{maintheoremsym}  tells us to look at the twist 6, thus we get a map $\mathcal{O}_X(2)^3\oplus \mathcal{O}_X(-1) \rightarrow \mathcal{O}_X(6)$. 
The global sections of $\mathcal{O}_X(2)^3\oplus \mathcal{O}_X(-1)$ are given by a tuple of three degree 2 elements and one degree -1 element. 
The degree $2$ elements have a monomial basis $z^2,zy,zx,y^2,yx,x^2$ with 6 generators.
The only element in negative degrees is $0$.
On the other hand, there are 28 monomials of degree 6.
Thus in the relevant twist $6$ we get a 28x18-matrix $B(6)$ with entries in $\C$, see Table \ref{Aqdmatrix}. 

\begin{table}[h]
\scalebox{0.6}[0.6]{
$\begin{blockarray}{rcccccc|cccccc|cccccc}
 & z^2 & zy& zx& y^2& yx& x^2 & z^2& zy& zx& y^2& yx& x^2 & z^2& zy& zx& y^2& yx& x^2 \\ 
\begin{block}{r(cccccc|cccccc|cccccc)}
z^6 && & & & & & & & & & & & 1& & & & & \\
z^5y && & & & & & & & & & & & & 1& & & & \\
z^5x && & & & & & & & & & & & & & 1& & & \\
z^4y^2 &1& & & & & & & & & & & & & & & 1& & \\
z^4yx && & & & & & & & & & & & & & & & 1& \\
z^4x^2 && & & & & & & & & & & & & & & & & 1\\
z^3y^3 && 1& & & & & & & & & & & & & & & & \\
z^3y^2x && & 1& & & & & & & & & & & & & & & \\
z^3yx^2 && & & & & & & & & & & & & & & & & \\
z^3x^3 && & & & & & & & & & & & & & & & & \\
z^2y^4 && & & 1& & & 1& & & & & & & & & & & \\
z^2y^3x && & & & 1& & & & & & & & & & & & & \\
z^2y^2x^2 && & & & & 1& & & & & & & & & & & & \\
z^2yx^3 && & & & & & & & & & & & & & & & & \\
z^2x^4 &1& & & & & & & & & & & & & & & & & \\
zy^5 && & & & & & & 1& & & & & & & & & & \\
zy^4x && & & & & & & & 1& & & & & & & & & \\
zy^3x^2 && & & & & & & & & & & & & & & & & \\
zy^2x^3 && & & & & & & & & & & & & & & & & \\
zyx^4 && 1& & & & & & & & & & & & & & & & \\
zx^5 && & 1& & & & & & & & & & & & & & & \\
y^6 && & & & & & & & & 1& & & & & & & & \\
y^5x && & & & & & & & & & 1& & & & & & & \\
y^4x^2 && & & & & & & & & & & 1& & & & & & \\
y^3x^3 && & & & & & & & & & & & & & & & & \\
y^2x^4 && & & 1& & & & & & & & & & & & & & \\
yx^5 && & & & 1& & & & & & & & & & & & & \\
x^6 && & & & & 1& & & & & & & & & & & & \\
\end{block}
\end{blockarray}$}

\caption{\label{Aqdmatrix}The matrix $B(6)$ from Example \ref{degreematrixexample}.
We have written the corresponding monomial basis elements on the top and to the left. 
0-entries have been omitted.}
\end{table}

\begin{table}[h]
\begin{tabular}{|r||l|l|l|l|l|} \hline
$q$ & $A_q$ & $k$ & $B_q(k)$ & $\dim \ker B_q(k)$ & $\Delta t$\\ \hline
1 & $1\times4$ & 6 & $28\times18$ & 0 & $<1$ms\\
2 & $4\times10$ & 12 & $156\times99$ & 0 & $<1$ms\\
3 & $10\times20$ & 18 & $501\times343$ & 0 & 2ms\\
4 & $20\times35$ & 25 & $1401\times1153$ & 0 & 8ms\\
5 & $35\times56$ & 31 & $2848\times2433$ & 0 & 30ms\\
6 & $56\times84$ & 37 & $5190\times4551$ & 0 & 98ms\\
7 & $84\times120$ & 44 &  $9474\times8763$ & 0 & 404ms\\
8 & $120\times165$ & 50 & $14889\times13891$ & 0 & 1s\\
9 & $165\times220$ & 56 & $22339\times20985$ & 0 &  4s\\
10 & $220\times286$ & 63 & $34219\times32865$ & 2 & 11s\\
11 & $286\times364$ & 69 & $47718\times45930$ & 0 & 27s\\
12 & $364\times455$ & 75 & $64845\times62551$ & 2& 69s\\	
13 & $455\times560$ & 82 & $90234\times88066$ & 128 & 169s\\
\vdots & & & & & \\
16 & $816\times969$ & 101 & $196743\times193608$ & 452 & 1743s \\ 
\vdots & & & & & \\
73 & $67525\times70300$ & 462 & very large & ? & ? \\ \hline
\end{tabular}
\bigskip
\caption{\label{matrixsizetable}This table accompanying Example \ref{degreematrixexample} lists the sizes of various matrices $A_q,$ the degree $k$ considered, and the size of the matrix $B_q(k)$. 
$\Delta t$ is the time to compute the kernel with our implementation on our computer. 
Since the actual computation time varies between computers and because there may be future optimizations these runtime numbers are only meant to illustrate the general trend.}
\end{table}

Consider Table \ref{matrixsizetable} for the sizes of the resulting matrices as $q$ grows. Note that $q=73$ is the power from Theorem \ref{maintheoremsym}.
\end{Example}

\begin{Remark}\label{sparsematricesremark}
The resulting matrices from Remark \ref{degreematrix} have an enormous size, prohibiting dense representations in computer memory.
Fortunately only few entries are nonzero.
If $A$ is an $m\times n$ matrix the matrix $A_q$ has only $n$ nonzero entries in each row, while it has $\binom{q+n-1}{n-1}$ columns. 
The matrix $B_q(k)$ is even more sparse assuming the polynomial entries of $A$ are sparse in the sense that they are made up of relatively few monomials compared to all monomials in their degree.
We can also see this phenomenon in the matrix of Table \ref{Aqdmatrix}.
There every column only has as many nonzero entries as the corresponding polynomial has nonzero coefficients.

To do any useful computations it is thus very important to store the matrices in a sparse matrix format.
This means that only the nonzero entries and their positions are stored.
This has not only the benefit of requiring less memory, it also means that we only need to iterate over the nonzero-entries in every reduction step.
We will perform matrix factorization - i.e. the process of factoring the matrix into triangle matrices and unitary matrices - in order to compute the kernel.
There is a lot of potential for optimization in the factorization of sparse matrices because naive implementations tend to introduce unnecessarily many additional nonzero entries making the matrix less sparse in the process.
It's important to reduce the rows in a good order and to choose good pivot elements.

There are extremely well optimized algorithms for sparse matrix factorization - but only for floating point values. 
One floating-point-algorithm we tried out is SuiteSparseQR\cite{Davis2011Algorithm9S}.
However, it does not seem practical (or maybe even possible) to control the cumulative floating point error in a way that let us with certainty distinguish a kernel of dimension 1 from a kernel of dimension 0.

Thus for an implementation of the algorithm we need to be able to do sparse exact value matrix triangularization.
Unfortunately most exact value matrix factorization implementations only work on dense matrices (for example it is implemented in Normaliz\cite{Normaliz}).
We considered using the sparse implementation in Bradford Hovinen's LELA\cite{LELA}, but it has only an optimized algorithm for matrices of the type occuring in Faugère's F4-algorithm and it proved difficult to use.

Because we didn't find a suitable implementation for integer matrix triangularization that suited our needs we implemented our own version of the Gauss Algorithm for sparse matrices in C++.
\end{Remark}

\begin{Remark}
For some of our results we need to work over an algebraically closed field.
Of course we can't actually represent complex numbers or any other uncountable field in computer memory. 
However if all involved coefficients in the input (which are the generators of the defining ideal $I$ of the base ring and the matrix $A$) are in $\Q$ any resulting kernel sections will also just have coefficients in $\Q$.
Even more, if the leading coefficients of a Gröbner basis of $I$ are units in $\Z$ and all entries of $A$ have coefficients only in $\Z$ we can do the whole kernel computation only with values in $\Z$.
Even though we then have to be careful in the Gauss Algorithm this is still faster and uses less memory than a representation in $\Q$.

Of course it would be possible to work with $\Q$ adjoint with a finite number of additional elements of $\C$, but we haven't explicitly implemented this.
Introducing an additional variable to the base polynomial ring and the necessary defining equations would be relatively easy but very costly.
\end{Remark}

\begin{Remark}
The performance of the algorithm depends a lot on the input. We list some of the characteristics and how they affect performance.
\begin{itemize}
\item Degree $d$ and genus $g$ of the curve $X=\Proj(S)$ are important in two ways.
Firstly, we have the symmetric power exponent $q=(g-1+d)n+1$ as of Theorem \ref{maintheoremsym}.
The higher $q$ is, the larger the matrix $A_q$ becomes, as $A_q$ is an $m\cdot \binom{q+l-2}{l-2}\times \binom{q+l-1}{l-1}$ matrix, where $m\times l$ are the dimensions of $A$.

Additionally $d$ and $g$ affect the size of the matrix $B_q(k)$: 
The Hilbert Polynomial of $X$ is $\Hilb(S)(t) = t\cdot d +(1-g)$. 
For large enough $t$ the values of $\Hilb(S)(t)$ are the dimensions of the vector spaces of degree $t$ elements of $S$.
Thus again larger $d$ are very costly here, while $g$ is only in the constant coefficient of the Hilbert polynomial and so does not affect the size as much.

\item The dimensions of the $m\times l$ matrix $A$ affect the size of $A_q$ directly as seen in the previous point. 
But it also goes into the rank and degree and thus the slope of $\mathcal{F}$.
Recall again Theorem \ref{maintheoremsym} and look at $n=\frac{r(r-1)}{\gcd(r,\deg\mathcal{F})}$, which multiplies into $q$.
At best this is $r-1$ and at worst $r(r-1)$.
Thus $r=l-m$ goes into the exponent $q$ linearly at best and quadratic in the worst case.

The twist is computed as $k=\left\lceil \frac{-q\mu(\mathcal{F})}{\deg X} \right\rceil - 1$. 
Here the degree of $\mathcal{F}$ enters, which is computed from the degrees of the entries of $A$.
High degree entries lead to a high twist.
A higher twist means higher degree of the monomials determining $B_q(k)$, thus there are more of them and $B_q(k)$ has a larger size.

\item The number of monomials used for the polynomials is also a strong factor. For the entries of $A$ this was explained in Remark \ref{sparsematricesremark}.

But for the Gröbner basis elements of $\mathfrak{a}$, where $\mathfrak{a}$ is the defining ideal of the curve, the same is true. 
If the result from a multiplication in the matrix $A_q$ is a leading monomial of an element $f$ in the reduced Gröbner basis, then it will be represented in $B_q(k)$ in the rows corresponding to all the other monomials of $f$.
If there are more monomials with nonzero coefficients then the matrix will be less sparse.
Thus we can say that the more ``general" the curve is, i.e. the more nonzero coefficients there are in the defining polynomials, the harder it is to compute with.

\item The coefficients of the polynomials involved also play a minor role.
Because of the way the integer Gauss algorithm works, during the triangularization the absolute values of the entries in $B_q(k)$ will generally increase a lot.
Because of this fixed length integer data types do not suffice as data containers and we need to use multiprecision integer data types, i.e. integer data types with an arbitrary length.
If we start with large values we will need even more digits to store the entries, which also increases the time used to handle them.
\end{itemize}

\end{Remark}

\begin{Remark}
For convenience of use, helpful future improvements to our implementation might include the following.
\begin{itemize}
\item An included feature to check whether the input ring is normal and to automatically work over the normalization if it isn't. 
So far we only check smoothness, and that the scheme is a curve at all, which are relatively easy to check.
\item The ability to embed any sheaf as a kernel sheaf automatically.
\item Further performance improvements to be able to check more and more difficult sheaves for semistability.
\end{itemize}
\end{Remark}

You will find our implementations of the main algorithm and the necessary subroutines together with an explanation on how to use them online at \url{https://github.com/JonathanSteinbuch/sheafstability}.

\bibliography{brennersteinbuchstabilityalgorithm}
\bibliographystyle{plain}
\end{document}